\documentclass[11pt,reqno]{amsart}
\usepackage[utf8]{inputenc}
\calclayout
\usepackage{amsmath}
\usepackage{amssymb,amsthm,xypic,stmaryrd,graphicx,mathtools}
\usepackage[all]{xy}
\usepackage{bbm, dsfont}
\usepackage{boxedminipage,xcolor}
\usepackage[shortlabels]{enumitem}
\usepackage{thmtools}
\usepackage{mathtools}
\usepackage{thm-restate}
\usepackage{tikz-cd}
\usepackage{empheq}
\usepackage{enumitem}
\usepackage[colorlinks=true,linkcolor=blue,citecolor=blue]{hyperref}
\usetikzlibrary{arrows.meta}
\tikzcdset{arrow style=tikz}
\numberwithin{equation}{section}
\DeclareMathOperator{\supp}{supp}
\DeclareMathOperator{\propagation}{prop}

\DeclareMathOperator{\Ind}{Ind}

\DeclareMathOperator{\Spin}{Spin}

\newcommand{\beq}[1]{\begin{equation} \label{#1}}
\newcommand{\eeq}{\end{equation}}

\newcommand{\bea}{\begin{eqnarray}}
\newcommand{\eea}{\end{eqnarray}}

\begin{document}

\theoremstyle{plain}
\newtheorem{theorem}{Theorem}[section]
\newtheorem{thm}{Theorem}[section]
\newtheorem{assumption}[theorem]{Assumption}
\newtheorem{lemma}[theorem]{Lemma}
\newtheorem{proposition}[theorem]{Proposition}
\newtheorem{prop}[theorem]{Proposition}
\newtheorem{corollary}[theorem]{Corollary}
\newtheorem{conjecture}[theorem]{Conjecture}
\newtheorem{question}[theorem]{Question}

\theoremstyle{definition}
\newtheorem{convention}[theorem]{Convention}
\newtheorem{definition}[theorem]{Definition}
\newtheorem{defn}[theorem]{Definition}
\newtheorem{example}[theorem]{Example}
\newtheorem{remark}[theorem]{Remark}
\newtheorem*{remark*}{Remark}
\newtheorem*{overview*}{Overview}
\newtheorem*{results*}{Results}
\newtheorem{rem}[theorem]{Remark}

\newcommand{\C}{\mathbb{C}}
\newcommand{\R}{\mathbb{R}}
\newcommand{\Z}{\mathbb{Z}}
\newcommand{\N}{\mathbb{N}}
\newcommand{\Q}{\mathbb{Q}}

\newcommand{\Supp}{{\rm Supp}}

\newcommand{\field}[1]{\mathbb{#1}}
\newcommand{\bZ}{\field{Z}}
\newcommand{\bR}{\field{R}}
\newcommand{\bC}{\field{C}}
\newcommand{\bN}{\field{N}}
\newcommand{\bT}{\field{T}}
\newcommand{\cB}{{\mathcal{B} }}
\newcommand{\cK}{{\mathcal{K} }}
\newcommand{\cF}{{\mathcal{F} }}
\newcommand{\cO}{{\mathcal{O} }}
\newcommand{\cE}{\mathcal{E}}
\newcommand{\cS}{\mathcal{S}}
\newcommand{\cN}{\mathcal{N}}
\newcommand{\calL}{\mathcal{L}}

\newcommand{\KK}{K \! K}

\newcommand{\norm}[1]{\| #1\|}

\newcommand{\Spinc}{\Spin^c}

\newcommand{\HH}{{\mathcal{H} }}
\newcommand{\Hpi}{\HH_{\pi}}

\newcommand{\DNR}{D_{N \times \R}}


\def\kt{\mathfrak{t}}
\def\kk{\mathfrak{k}}
\def\kp{\mathfrak{p}}
\def\kg{\mathfrak{g}}
\def\kh{\mathfrak{h}}
\def\so{\mathfrak{so}}
\def\cut{c}

\newcommand{\ddt}{\left. \frac{d}{dt}\right|_{t=0}}
\newcommand{\todoa}[1]{\color{red}\textbf{TO DO$^A$: }{#1 }\color{black}}
\newcommand{\todob}[1]{\color{teal}\textbf{TO DO$^B$: }{#1 }\color{black}}
\newcommand{\todoc}[1]{\color{olive}\textbf{TO DO$^C$: }{#1 }\color{black}}
\newcommand{\todod}[1]{\color{blue}\textbf{TO DO$^C$: }{#1 }\color{black}}
\newcommand{\todoe}[1]{\color{green}\textbf{TO DO$^C$: }{#1 }\color{black}}

\newcommand{\esta}{\color{red}\textbf{estimate $1$}\color{black}}
\newcommand{\estb}{\color{teal}\textbf{estimate $2$}\color{black}}
\newcommand{\estc}{\color{olive}\textbf{estimate $3$}\color{black}}
\newcommand{\estd}{\color{blue}\textbf{estimate $4$}\color{black}}
\newcommand{\este}{\color{green}\textbf{estimate $5$}\color{black}}

\title{Quantitative K-theory, positive scalar curvature, and band width}

\author{Hao Guo}
\address[Hao Guo]{ Department of Mathematics, Texas A\&M University }
\email{haoguo@math.tamu.edu}
\author{Zhizhang Xie}
\address[Zhizhang Xie]{ Department of Mathematics, Texas A\&M University }
\email{xie@math.tamu.edu}

\author{Guoliang Yu}
\address[Guoliang Yu]{ Department of
	Mathematics, Texas A\&M University}
\email{guoliangyu@math.tamu.edu}

\thanks{H.G. is partially supported by NSF DMS-2000082. Z.X. is partially supported by NSF DMS-1800737 and NSF DMS-1952693. G.Y. is partially supported by NSF DMS-1700021, NSF DMS-2000082, and the Simons Fellows Program}


\maketitle

\begin{abstract}
We develop two connections between the quantitative framework of operator $K$-theory for geometric $C^*$-algebras and the problem of positive scalar curvature. First, we introduce a quantitative notion of higher index and use it to give a refinement of the well-known obstruction of Rosenberg to positive scalar curvature on closed spin manifolds coming from the higher index of the Dirac operator. We show that on a manifold with uniformly positive scalar curvature, the propagation at which the index of the Dirac operator vanishes is related inversely to the curvature lower bound. Second, we give an approach, using related techniques, to Gromov's band width conjecture, which has been the subject of recent work by Zeidler and Cecchini from a different point of view.
\end{abstract}
\vspace{1cm}
\tableofcontents

\section{Introduction and motivation}
\label{sec intro}
Many questions in topology and geometry have important links to the $K$-theory of $C^*$-algebras, which in particular is the receptacle for a vast generalization of classical index theory, known as \emph{higher index theory}. For example, when an elliptic differential operator on a closed manifold is lifted to the universal cover, the \emph{higher index} of the lifted operator can be constructed by taking into account the action of the fundamental group \cite{Baum-Connes,BCH,ConnesNCG,Kasparov,WillettYu}. The higher index plays a fundamental role in the study of geometry and topology through the Novikov conjecture \cite{CM90,Kasparov,Yu1,Yu2} on homotopy invariance of higher signatures and the Gromov-Lawson conjecture \cite{Gromov-Lawson,Gromov-Lawson2} on the existence of Riemannian metrics with positive scalar curvature, while the Baum-Connes Conjecture \cite{Baum-Connes,BCH} proposes an algorithm for computing the higher index.

\emph{Quantitative $K$-theory} is a framework for computing $K$-theory by exploiting underlying geometric structures present in a $C^*$-algebra. It is a refinement of ordinary operator $K$-theory in that the latter can be realized as a certain limit of quantitative $K$-groups. This motivates one to consider quantitative $K$-theoretic refinements of various invariants and obstructions that occur in higher index theory.

The first part of this paper is concerned with generalizing the higher index to this new setting. The resulting \emph{quantitative higher index} provides a refinement of the well-known index-theoretic obstruction of Rosenberg \cite{Rosenberg1} to the existence of positive scalar curvature on spin manifolds. We prove:
\begin{theorem}
\label{thm main}
	Let $M$ be a Riemannian spin manifold with fundamental group $\Gamma$. Let $\kappa$ be the scalar curvature and $D_{\widetilde{M}}$ the lift of the Dirac operator on $M$ to its universal cover $\widetilde{M}$. Fix $0<\varepsilon<\frac{1}{20}$. There exists a constant $\omega_0$ such that for every $c>0$, if $\kappa\geq c$ uniformly on $M$ then the $\Gamma$-equivariant quantitative maximal higher index of $D_{\widetilde M}$ at scale $r$ vanishes for all $r\geq\frac{\omega_0}{\sqrt{c}}$:
	$$\Ind_{\Gamma,\textnormal{max}}^{\varepsilon,r,N}(D_{\widetilde{M}})=0\in K_*^{\varepsilon,r,N}(C^*_\textnormal{max}(\widetilde M)^{\Gamma}),$$
	for any $N\geq 7$. The constant $\omega_0$ is independent of the manifold $M$. (See section \ref{sec quantitative} for the definition of the quantitative maximal higher index.)
	\end{theorem}

In the second part of this paper, we prove a result on Gromov's band width conjecture (see \cite[11.12, Conjecture C]{Gromov18} and \cite[section 3.6]{GromovLectures}) using tools from quantitative $K$-theory. The conjecture is as follows.
\begin{conjecture}
\label{conj band width}
	Let $M$ be a closed manifold of dimension $n-1\geq 5$ that does not admit a Riemannian metric of positive scalar curvature. Then there exists a constant $C_n$ such that for every Riemannian manifold $V$ diffeomorphic to $M\times[-1,1]$ with scalar curvature bounded below by $\sigma>0$, we have
	\begin{equation}
	\label{eq Gromov}
	L\coloneqq\textnormal{dist}(\partial_-V,\partial_+V)\leq\frac{C_n}{\sqrt{\sigma}},
	\end{equation}
	where $\partial_\pm V=M\times\{\pm 1\}$, and $\textnormal{dist}$ is the Riemannian distance.
\end{conjecture}

This conjecture has been actively studied in recent work by Zeidler \cite{Zeidler, Zeidler2} and Cecchini \cite{CecchiniLongNeck}, who have succeeded in attaining the optimal constant $C_n=2\pi\sqrt{\frac{n-1}{n}}$ when $M$ is spin and has non-vanishing Rosenberg index.\footnote{In fact, the inequality \eqref{eq Gromov} has been shown to hold strictly with this constant -- see \cite[Corollary 1.5]{Zeidler}.} Our purpose here is to provide an alternative approach to the band width conjecture using tools from quantitative $K$-theory. We prove:

\begin{theorem}
\label{thm main 2}
Let $M$ be a closed spin manifold of dimension $n-1$ with fundamental group $\Gamma$. Let $D_{\widetilde{M}}$ be the lift of the Dirac operator on $M$ to its universal cover $\widetilde{M}$. Then there exists a constant $C$, independent of $M$, such that for every Riemannian manifold $V$ diffeomorphic to $M\times[-1,1]$ with scalar curvature bounded below by $\sigma>0$, if
	\[ L > C\sqrt{\frac{n-1}{n}} \frac{1}{\sqrt{\sigma}}, \]
	then 
	\[  \Ind_{\Gamma,\textnormal{max}}(D_{\widetilde M}) = 0 \in K_{1}(C_{\max}^\ast( \Gamma)).\]
If $n$ is even, then $C\leq 190\pi$, and if $n$ is odd, then $C\leq 328\pi$.
In the geometric setting of Conjecture \ref{conj band width}, we may take the constant $C_n$ in \eqref{eq Gromov} to be $C\sqrt{\frac{n-1}{n}}$.
\end{theorem}
\begin{remark}
\label{rem general band}
	The method of proof for Theorem \ref{thm main 2} also applies to the more general situation of a \emph{proper band} \cite{Gromov18}, i.e. a manifold with two distinguished subsets $\partial_+$ and $\partial_-$ of the boundary, each of which is a union of connected components of the boundary. Here, given the same inequality between scalar curvature and the distance between $\partial_+$ and $\partial_-$, the conclusion is that index of the Dirac operator on $\partial_-$ vanishes.
\end{remark}
\begin{remark}
Using computer assistance, Xie and Wang \cite[Appendix B]{Zhizhang} have recently improved the constant $C$ in Theorem \ref{thm main 2} to $\sim 64\pi$.
\end{remark}

\begin{overview*}
\hfill\vskip 0.05in
\noindent The paper is organized as follows. We begin in section \ref{sec prelim} by recalling some standard geometric and operator-algebraic terminology. In section \ref{sec quantitative}, we review quantitative $K$-theory and define the quantitative higher index. In section \ref{sec main theorem} we prove Theorem \ref{thm main}, which generalizes the Lichnerowicz vanishing theorem for higher indices. Finally, in section \ref{sec band width}, we provide an approach to Gromov's band width conjecture using related ideas, and prove Theorem \ref{thm main 2}.
\end{overview*}
\vspace{0.3in}
\section{Preliminaries}
\label{sec prelim}
We first fix some notation and recall the necessary operator-algebraic and geometric terminology we will need.
\vspace{0.1in}
\subsection{Notation}
\label{subsec notation}
\hfill\vskip 0.05in
	\noindent 
	For $X$ a Riemannian manifold, we write $B(X)$, $C_b(X)$, $C_0(X)$, and $C_c(X)$ to denote the $C^*$-algebras of complex-valued functions on $X$ that are, respectively: bounded Borel, bounded continuous, continuous and vanishing at infinity, and continuous with compact support. If $S\subseteq X$ is a Borel subset, we write $\mathbbm{1}_S$ for the associated characteristic function.
	
	
	For any $C^*$-algebra $A$, denote its unitization by $A^+$ and its multiplier algebra by $\mathcal{M}(A)$. We view $A$ as an ideal of $\mathcal{M}(A)$. 
	
	
	The action of a group $G$ on $X$ naturally induces a $G$-action on spaces of functions on $X$ as follows: given a function $f$ on $X$ and $g\in G$, define $g\cdot f$ by $g\cdot f(x)=f(g^{-1}x)$. More generally, for a section $s$ of a $\Gamma$-vector bundle over $X$, the section $g\cdot s$ is defined by $g\cdot s(x)=g(s(g^{-1}x))$. We say that an operator on sections of a bundle is $G$-equivariant if it commutes with the $G$-action.
	
	We will denote the maximal group $C^*$-algebra of a group $G$ by $C^*_{\textnormal{max}}(G)$.
	

\hfill\vskip 0.1in
\subsection{Geometric $C^*$-algebras and modules}
	\label{subsec op alg}
 	\hfill\vskip 0.05in
	\noindent First let us recall the general notion of a geometric $C^*$-algebra.
\begin{definition}
\label{def geometric algebra}
A $C^*$-algebra $A$ is said to be \emph{geometric} if it admits a filtration $\{A_r\}_{r>0}$ satisfying the following properties:
\begin{enumerate}[(i)]
\item $A_r\subseteq A_{r'}$ if $r\leq r'$;
\item $A_r A_{r'}\subseteq A_{r+r'}$;
\item $\bigcup_{r=0}^\infty A_r$ is dense in $A$.
\end{enumerate}
An element $a\in A_r$ is said to have \emph{propagation at most $r$}, for which we write
$$\textnormal{prop}(a)\leq r.$$
\end{definition}
If $A$ is non-unital, then its unitization $A^+$, viewed as $A\oplus\mathbb{C}$ as as a vector space, is a geometric $C^*$-algebra with filtration
$$\{A_r\oplus\mathbb{C}\}_{r>0}.$$ 
In addition, for each $n$, the matrix algebra $M_n(A)$ is a geometric $C^*$-algebra with filtration
$$\{M_n(A_r)\}_{r>0}.$$

The particular example of a geometric $C^*$-algebra that we will be interested in is the (maximal) Roe algebra associated to a geometric module, which we now review. 

For the rest of this section, let $X$ be a Riemannian manifold on which a discrete group $G$ acts properly and isometrically. 

First recall the following fact: if $H$ is a Hilbert space and $\rho\colon C_0(X)\rightarrow\mathcal{B}(H)$ is a non-degenerate $*$-representation, then $\rho$ extends uniquely to a $*$-representation $\widetilde{\rho}\colon B(X)\rightarrow\mathcal{B}(H)$ subject to the property that, for a uniformly bounded sequence in $B(X)$ converging pointwise, the corresponding sequence in $\mathcal{B}(H)$ converges in the strong topology.

\begin{definition}
\label{def XGmodule}
	An \emph{admissible $X$-$G$-module} is a separable Hilbert space $\mathcal{H}$ equipped with a non-degenerate $*$-representation $\rho\colon C_0(X)\rightarrow\mathcal{B}(\mathcal{H})$ and a unitary representation $U\colon G\rightarrow\mathcal{U}(\mathcal{H})$ such that:
	\begin{enumerate}[(i)]
		\item for all $f\in C_0(X)$ and $g\in G$, we have  $U_g\rho(f)U_g^*=\rho(g\cdot f)$;
		\item for any non-zero $f\in C_0(X)$ we have $\rho(f)\notin\mathcal{K}(\mathcal{H})$;
		\item for any finite subgroup $F$ of $ G$ and any $F$-invariant Borel subset $E\subseteq X$, there is a Hilbert space $H'$ equipped with the trivial $F$-representation such that $\widetilde{\rho}(\mathbbm{1}_E)H'\cong l^2(F)\otimes H'$ as $F$-representations, where $\widetilde{\rho}$ is defined by extending $\rho$ as above.
	\end{enumerate}
\end{definition}
	For brevity, we will omit $\rho$ from the notation when it is clear from context.
	\begin{definition}
	\label{def:suppprop}
		Let $\mathcal{H}$ be an admissible $X$-$G$-module and $T\in\mathcal{B}(\mathcal{H})$.
		\begin{itemize}
			\item The \emph{support} of $T$, denoted $\textnormal{supp}(T)$, is the complement of all $(x,y)\in X\times X$ for which there exist $f_1,f_2\in C_0(X)$ such that $f_1(x)\neq 0$, $f_2(y)\neq 0$, and
			$$f_1Tf_2=0;$$
			\item The \textit{propagation} of $T$ is the extended real number $$\textnormal{prop}(T)=\sup\{d_X(x,y)\,|\,(x,y)\in\textnormal{supp}(T)\};$$
			\item $T$ is \textit{locally compact} if $fT$ and $Tf\in\mathcal{K}(\mathcal{H})$ for all $f\in C_0(X)$;
			\item $T$ is \textit{$ G$-equivariant} if $U_g TU_g^*=T$ for all $g\in G$;
		\end{itemize}
		The \emph{$G$-equivariant algebraic Roe algebra of $X$} is the $*$-subalgebra of $\mathcal{B}(\mathcal{H})$ consisting of $ G$-equivariant, locally compact operators with finite propagation, and is denoted by $\mathbb{C}[X]^G$.
		\end{definition}
\begin{remark}
	The notation is justified by the fact that the $*$-algebra $\mathbb{C}[X]^G$ is independent of the choice of admissible $X$-$G$-module -- see \cite[Chapter 5]{WillettYu}.
\end{remark}
		\begin{definition}
			\label{def:maximalnorm}
			The \emph{maximal norm} of an operator $T\in\mathbb{C}[X]^{ G}$ is
			$$||T||_{\textnormal{max}}\coloneqq\sup_{\phi,H'}\left\{\norm{\phi(T)}_{\mathcal{B}(H')}\,|\,\phi\colon\mathbb{C}[X]^{ G}\rightarrow\mathcal{B}(H')\textnormal{ is a $*$-representation}\right\}.$$
			The \emph{maximal equivariant Roe algebra of $X$}, denoted $C^*_{\text{max}}(X)^ G$, is the completion of $\mathbb{C}[X]^ G$ in the norm $||\cdot||_{\textnormal{max}}$.
		\end{definition}
		\begin{remark}
		\label{rem maximal Roe}
		If $G$ acts on $X$ freely, properly and cocompactly, then there is a $*$-isomorphism
		$$C_{\textnormal{max}}^*(X)^G\cong C^*_{\textnormal{max}}(G)\otimes\mathcal{K},$$
		where $\mathcal{K}$ is the algebra of compact operators on a separable infinite-dimensional Hilbert space. To make sense of Definition \ref{def:maximalnorm} for more general $X$ and $G$, one first needs to establish finiteness of the quantity $\norm{\cdot}_{\textnormal{max}}$. In \cite{GXY} this was shown to be the case when $X$ has bounded geometry and the $G$-action satisfies a suitable geometric assumption, which is in particular satisfied by the geometric situations considered in this paper (see Assumption \ref{ass:condition}).
		\end{remark}
The maximal equivariant Roe algebra $C_{\textnormal{max}}^*(X)^G$ is a geometric $C^*$-algebra in the sense of Definition \ref{def geometric algebra} with respect to the filtration $\{\mathbb{C}[X]^G_r\}_{r>0}$, where
\begin{equation}
\label{eq filtration}
\mathbb{C}[X]^G_r\coloneqq\{T\in\mathbb{C}[X]^G\colon\textnormal{prop}(T)\leq r\}.
\end{equation}
\hfill\vskip 0.3in
\section{Quantitative K-theory and higher index}
\label{sec quantitative}
In this section, we review quantitative $K$-theory and define a refinement of the higher index, called the quantitative higher index. 

We begin by reviewing quantitative $K$-groups. We take the approach in \cite{Chung} using quasiidempotents and quasiinvertibles (compare \cite{Oyono2}, where quasiprojections and quasiunitaries are used). Doing this allows us to better control the propagation of index representatives, which, in the even-dimensional setting, are more naturally given by idempotents than by projections.
\subsection{Quantitative K-theory}
\label{subsec qkt} 
\begin{definition}[{\cite[Definition 2.15]{Chung}}]
\label{def quasi}
Let $A$ be a geometric $C^*$-algebra. For $0<\varepsilon<\frac{1}{20}$, $r>0$, and $N\geq 1$,
\begin{itemize}
\item an element $e\in A$ is called an \emph{$(\varepsilon,r,N)$-quasiidempotent} if
$$\norm{e^2-e}<\varepsilon,\qquad e\in A_r,\qquad\max(\norm{e},\norm{1_{A^+}-e})\leq N;$$
\item if $A$ is unital, an element $u\in A$ is called an \emph{$(\varepsilon,r,N)$-quasiinvertible} if $u\in A_r$, $\norm{u}\leq N$, and there exists $v\in A_r$ with
$$\norm{v}\leq N,\qquad\max(\norm{uv-1},\norm{vu-1})<\varepsilon.$$
The pair $(u,v)$ is called an $(\varepsilon,r,N)$\emph{-quasiinverse pair}.
\end{itemize}
\end{definition}
The quantitative $K$-groups $K_0^{\varepsilon,r,N}(A)$ and $K_1^{\varepsilon,r,N}(A)$ are defined by collecting together all quasiidempotents and quasiinvertibles over all matrix algebras, quotienting by an equivalence relation, and taking the Gr\"othendieck completion.
\begin{definition}[{\cite[subsection 3.1]{Chung}}]
\label{def quantitative K}
Let $A$ be a unital geometric $C^*$-algebra. Let $r>0$, $0<\varepsilon<\frac{1}{20}$, and $N>0$. 
\begin{enumerate}[leftmargin=0.29in]
\item Denote by $\textnormal{Idem}^{\varepsilon,r,N}(A)$ the set of $(\varepsilon,r,N)$-quasi-idempotents in $A$. For each positive  integer $n$, let
$$\textnormal{Idem}_n^{\varepsilon,r,N}(A)=\textnormal{Idem}^{\varepsilon,r,N}(M_n(A)).$$
We have inclusions $\textnormal{Idem}_n^{\varepsilon,r,N}(A)\hookrightarrow\textnormal{Idem}_{n+1}^{\varepsilon,r,N}(A)$ given by
$e\mapsto\begin{pmatrix}e&0\\0&0\end{pmatrix}.$ Set 
$$\textnormal{Idem}_\infty^{\varepsilon,r,N}(A)=\bigcup_{n=1}^\infty\textnormal{Idem}_n^{\varepsilon,r,N}(A).$$
Define an equivalence relation $\sim$ on $\textnormal{Idem}_\infty^{\varepsilon,r,N}(A)$ by $e\sim f$ if $e$ and $f$ are $(4\varepsilon,r,4N)$-homotopic in $M_\infty(A)$. Denote the equivalence class of an element $e\in\textnormal{Idem}_\infty^{\varepsilon,r,N}(A)$ by $[e]$. Define addition on $\textnormal{Idem}_\infty^{\varepsilon,r,N}(A)/\sim$ by
$$[e]+[f]=\begin{bmatrix}e&0\\0&f\end{bmatrix}.$$
With this operation, $\textnormal{Idem}_\infty^{\varepsilon,r,N}(A)/\sim$ is an abelian monoid with identity $[0]$. Let $K_0^{\varepsilon,r,N}(A)$ denote its Grothendieck completion.

\item Denote by $GL^{\varepsilon,r,N}(A)$ the set of $(\varepsilon,r,N)$-quasiinvertibles in $A$. For each positive integer $n$, let
$$GL_n^{\varepsilon,r,N}(A)=GL_n^{\varepsilon,r,N}(M_n(A)).$$
We have inclusions $GL_n^{\varepsilon,r,N}(A)\hookrightarrow GL_{n+1}^{\varepsilon,r,N}(A)$ given by
$u\mapsto\begin{pmatrix}u&0\\0&1\end{pmatrix}.$ Set 
$$GL_\infty^{\varepsilon,r,N}(A)=\bigcup_{n=1}^\infty GL_n^{\varepsilon,r,N}(A).$$
Define an equivalence relation $\sim$ on $GL_\infty^{\varepsilon,r,N}(A)$ by $e\sim f$ if $u$ and $v$ are $(4\varepsilon,2r,4N)$-homotopic in $M_\infty(A)$. Denote the equivalence class of an element $u\in GL_\infty^{\varepsilon,r,N}(A)$ by $[u]$. Define addition on $GL_\infty^{\varepsilon,r,N}(A)/\sim$ by
$$[u]+[v]=\begin{bmatrix}u&0\\0&v\end{bmatrix}.$$
With this operation, $GL_\infty^{\varepsilon,r,N}(A)/\sim$ is an abelian group with identity $[1]$.
\end{enumerate}
\end{definition}
\begin{remark}
\label{rem nonunital}	
If $A$ is a non-unital geometric $C^*$-algebra, then we have a canonical $*$-homomorphism $\pi\colon A^+\to\mathbb{C}$. Using contractivity of $\pi$, we have homomorphisms
$$\pi_*\colon K_i^{\varepsilon,r,N}(A^+)\to K_i^{\varepsilon,r,N}(\mathbb{C}),$$
where $i=0$ or $1$. Define $K_i^{\varepsilon,r,N}(A)=\ker(\pi_*)$.
\end{remark}
We have the following useful result on quasiidempotents and quasiinvertibles:
\begin{lemma}
\label{lem htpy}
Let $A$ be a geometric $C^*$-algebra. If $e$ is an $(\varepsilon,r,N)$-idempotent in $A$, and $f\in A_r$ satisfies 
$$\norm{f}\leq N,\qquad\norm{e-f}<\frac{\varepsilon-\norm{e^2-e}}{2N+1},$$ 
then $f$ is a quasiidempotent that is $(\varepsilon,r,N)$-homotopic to $e$. In particular, if
$$\norm{f}<\frac{\varepsilon}{2N+1},$$
then the class of $f$ is zero in $K_0^{\varepsilon,r,N}(A)$.

Suppose that $A$ is unital and $(u,v)$ is an $(\varepsilon,r,N)$-quasiinverse pair in $A$. If $a\in A_r$ satisfies 
$$\norm{a}\leq N,\qquad\norm{u-a}<\frac{\varepsilon-\max(\norm{uv-1},\norm{vu-1})}{N},$$ 
then $a$ is a quasiinvertible that is $(\varepsilon,r,N)$-homotopic to $u$. In particular, if
$$\norm{1-a}<\frac{\varepsilon}{N},$$
then the class of $a$ is zero in $K_1^{\varepsilon,r,N}(A)$.
\end{lemma}
\begin{proof}
See \cite[Lemma 2.19]{Chung}, noting that a geometric $C^*$-algebra is in particular a filtered Banach algebra in the sense of \cite{Chung}.
\end{proof}

For $i=0$ or $1$, there is a homomorphism of abelian groups
$$\kappa_i\colon K_i^{\varepsilon,r,N}(A)\to K_i(A)$$
mapping from quantitative $K$-theory to ordinary $K$-theory defined as follows (see \cite[subsection 3.3]{Chung}).

Suppose first that $A$ is unital. Let $e$ be an $(\varepsilon,r,N)$-quasiidempotent in $A$, the spectrum of $e$ is contained in the union of disjoint balls $B_{\sqrt{\varepsilon}}(0)\cup B_{\sqrt{\varepsilon}}(1)\subseteq\mathbb{C}$. Choose a function $f_0$ that is holomorphic on a neighborhood of the spectrum and such that
$$
f_0(z)\equiv
\begin{cases}
0&\textnormal{ if }z\in\overline{B}_{\sqrt{\varepsilon}}(0),\\
1&\textnormal{ if }z\in\overline{B}_{\sqrt{\varepsilon}}(1).
\end{cases}
$$
Let $\gamma$ be the contour
$$\{z\in\mathbb{C}\colon|z|=\sqrt{\varepsilon}\}\cup\{z\in\mathbb{C}\colon|z-1|=\sqrt{\varepsilon}\}.$$
Using the holomorphic functional calculus, we get an element
\begin{equation}
\label{eq f0}	
f_0(e)=\frac{1}{2\pi i}\int_{\gamma}f_0(z)(z-e)^{-1}\,dz,
\end{equation}
which one can verify is an idempotent in $A$.

More generally, if $A$ is not necessarily unital, one may apply this procedure to $(\varepsilon,r,N)$-quasiidempotents in matrix algebras over $A^+$. Further, \cite[Proposition 3.19]{Chung} shows that $[e]\mapsto[f_0(e)]$ gives a well-defined group homomorphism
$$\kappa_0\colon K_0^{\varepsilon,r,N}(A)\to K_0(A).$$
Next, observe that if $u$ is an $(\varepsilon,r,N)$-quasiinvertible, then $u$ is in fact invertible, since by Definition \ref{def quasi} there exists $v\in A_r$ such that
$\norm{v}\leq N$, $\norm{uv-1}<\varepsilon$, and $\norm{vu-1}<\varepsilon$. In particular, $uv$ and $vu$ are both invertible, hence $u$ is invertible.

From this we obtain a group homomorphism
$$\kappa_1\colon K_1^{\varepsilon,r,N}(A)\to K_1(A).$$
We will denote the direct sum of $\kappa_0$ and $\kappa_1$ by
\begin{equation}
\label{eq kappa}	
\kappa\colon K_*^{\varepsilon,r,N}(A)\to K_*(A),
\end{equation}
where $K_*^{\varepsilon,r,N}(A)\coloneqq K_0^{\varepsilon,r,N}(A)\oplus K_1^{\varepsilon,r,N}(A)$ and $K_*(A)\coloneqq K_0(A)\oplus K_1(A)$.

We note that the idempotent $f_0(e)$ constructed in \eqref{eq f0} satisfies the following estimates, which we will use later:
\begin{proposition}[{\cite[Proposition 3.18]{Chung}}] Let $e$ be an $(\varepsilon,r,N)$-quasiidempotent in $A$, and let $f_0(e)$ be as in \eqref{eq f0}. Then
\label{prop regularization}
	\begin{align}
	\label{eq hol idem general}
	\norm{f_0(e)}&<\frac{N+1}{1-2\sqrt{\varepsilon}},\nonumber\\
	\norm{f_0(e)-e}&<\frac{2(N+1)\varepsilon}{(1-\sqrt{\varepsilon})(1-2\sqrt{\varepsilon})}.
	\end{align}
\end{proposition}

\vspace{0.1in}
\subsection{The quantitative higher index}
\label{subsec quantitative higher ind}
\hfill\vskip 0.05in
	\noindent The quantitative higher index we now define is a refinement of the usual higher index that retains information about the propagation of the index representative. To do this, let us first recall the definition of the usual higher index.
	
\subsubsection{Higher index}
\label{subsubsec higher ind}
\hfill\vskip 0.1in
\noindent 
A short exact sequence of $C^*$-algebras
	$$0\rightarrow I\rightarrow A\rightarrow A/I\rightarrow 0,$$
	induces a six-term exact sequence in $K$-theory:
	\[
	\begin{tikzcd}
	K_0(I) \ar{r} & K_0(A) \ar{r} & K_0(A/I) \ar{d}{\partial_1} \\
	K_1(A/I) \ar{u}{\partial_0} & K_1(A) \ar{l} & K_1(I) \ar{l},
	\end{tikzcd}
	\]
	where the connecting maps $\partial_0$ and $\partial_1$ are defined as follows.
	\begin{definition}
		\label{def connectingmaps}
		\hfill
		\begin{enumerate}[(i)]
			\item $\partial_0$: let $u$ be an invertible matrix with entires in $A/I$ representing a class in $K_1(A/I)$.
			Write
			\[w=
			\begin{pmatrix}
			0&-u^{-1}\\u&0	
			\end{pmatrix}=
			\begin{pmatrix} 
			1 & 0\\ 
			u & 1
			\end{pmatrix}
			\begin{pmatrix} 
			1 & -u^{-1}\\ 
			0 & 1
			\end{pmatrix}
			\begin{pmatrix} 
			1 & 0\\ 
			u & 1
			\end{pmatrix}.
			\]
			Then $w$ lifts to an invertible matrix $W$ with entries in $A$	. Then
			$$
			P=W\begin{pmatrix}
			1 & 0\\ 
			0 & 0
			\end{pmatrix}W^{-1}
			$$
			is an idempotent, and we define
			\begin{equation}
			\label{eq even index}
			\partial_0[u]\coloneqq 
			\left[P
			\right]-
			\begin{bmatrix}
			0 & 0\\
			0 & 1
			\end{bmatrix}
			\in K_0(I).
			\end{equation}

			\item $\partial_1$: let $q$ be an idempotent matrix with entries in $A/I$ representing a class in $K_0(A/I)$. Let $Q$ be a lift of $q$ to a matrix algebra over $A$. Then we define
			\begin{equation}
			\label{eq odd index}
			\partial_1[q]\coloneqq\left[e^{2\pi iQ}\right]\in K_1(I).
			\end{equation}
		\end{enumerate}
	\end{definition}
	
	Now let $(M,g)$ be a Riemannian manifold equipped with a proper isometric action by a group $\Gamma$. Let $D$ be a $\Gamma$-equivariant first-order essentially self-adjoint elliptic differential operator on a bundle $E\to M$. We will assume throughout that if $M$ is odd-dimensional then $D$ is an ungraded operator, while if $M$ is even-dimensional then $D$ is odd-graded with respect to a $\mathbb{Z}_2$-grading on $E$. Let $\mathcal{M}$ be the multiplier algebra of $C^*_\textnormal{max}(M)^{\Gamma}$, and let $\mathcal{Q}=\mathcal{M}/C^*_\textnormal{max}(M)^{\Gamma}.$ We have a short exact sequence of $C^*$-algebras
	$$0\rightarrow C^*_\textnormal{max}(M)^{\Gamma}\rightarrow\mathcal{M}\rightarrow\mathcal{Q}\rightarrow 0.$$ Choose a \emph{normalizing function} $\chi\colon\mathbb{R}\rightarrow\mathbb{R}$, i.e. a continuous, odd function such that 
	$$\lim_{x\rightarrow +\infty}\chi(x)=1.$$

	Using the functional calculus for the maximal Roe algebra from \cite{GXY}, we may form the bounded adjointable operator $\chi(D)$ on the Hilbert module $C^*_{\textnormal{max}}(M)^\Gamma$ over itself.
When $\dim M$ is even, we can write 
$$\chi(D)=\begin{pmatrix}0&\chi(D)^-\\\chi(D)^+&0\end{pmatrix}.$$
	It follows from \cite[Proposition 4.1]{GXY} that the class of $\chi(D)^+$ in $\mathcal{M}/C^*_\textnormal{max}(M)^{\Gamma}$ is invertible and independent of the choice of $\chi$,
	while the class of $\frac{\chi(D)+1}{2}$ is an idempotent. This leads us to the definition of the maximal higher index of $D$:
\begin{definition}
\label{def higher index}
			For $i=0,1$, let $\partial_i$ be the connecting maps from Definition \ref{def connectingmaps}. The \emph{maximal higher index} of $D$ is the element
			\begin{empheq}[left={\Ind_{\Gamma,\textnormal{max}}D\coloneqq
\empheqlbrace}]{alignat*=2}
    \partial_{0}\left[\chi(D)^+\right]\in K_{0}\big(C^*_\textnormal{max}(M)^{\Gamma}\big)\quad&\textnormal{ if $\dim M$ is even},\\[1.5ex]
    \partial_{1}\left[\tfrac{\chi(D)+1}{2}\right]\in K_{1}\big(C^*_\textnormal{max}(M)^{\Gamma}\big)\quad&\textnormal{ if $\dim M$ is odd}.
\end{empheq}
\end{definition}

We have the following explicit representatives for the index. For $\dim M$ even, let
\begin{equation}
\label{eq pchi}
p_{\chi}(D)=\begin{pmatrix}
			\left[(1-\chi(D)^2)^2\right]_{1,1} & \,\,\left[\chi(D)(1-\chi(D)^2)\right]_{1,2}\\[1ex]
			\left[\chi(D)(2-\chi(D)^2)(1-\chi(D)^2)\right]_{2,1} & \,\,\left[\chi(D)^2(2-\chi(D)^2)\right]_{2,2}
		\end{pmatrix},
\end{equation}
where the notation $[X]_{i,j}$ means the $(i,j)$-th entry of the matrix $X$. Then $p_\chi(D)$ is an idempotent matrix, and  $\Ind_{\Gamma,\textnormal{max}}D$ is represented by
	\begin{equation}
	\label{eq Aj graded}
	A_\chi(D)=p_{\chi}(D)-\begin{pmatrix}0&0\\0&1\end{pmatrix}.
	\end{equation}
For $\dim M$, $\Ind_{\Gamma,\textnormal{max}}D$ can be represented by the unitary
	\begin{equation}
	\label{eq Aj ungraded}
	A_\chi(D)=e^{\pi i(\chi+1)}(D).
	\end{equation}
\vspace{0.1in}
\subsubsection{Quantitative higher index}
\label{sec qhi}
\hfill\vskip 0.05in
\noindent The maximal equivariant Roe algebra is a geometric $C^*$-algebra with respect to the filtration \eqref{eq filtration}. This allows us to define a version of the maximal higher index that lives in quantitative $K$-theory using finite-propagation representatives. We now do this for even and odd-dimensional $M$ separately.

Throughout this subsection, fix 
	$$0<\varepsilon<\tfrac{1}{20},\quad r>0,\quad N\geq 7.$$

\subsubsection*{Even-dimensional case}
\hfill\vskip 0.05in
	\noindent Suppose $M$ is even-dimensional.
	Choose a normalizing function $\chi$ whose distributional Fourier transform satisfies 
	\begin{equation}
	\label{eq fourier r5}
	\supp\widehat{\chi}\subseteq\left[-\frac{r}{5},\frac{r}{5}\right].
	\end{equation}
	Let $A_{\chi}(D)=p_{\chi}(D)-\begin{psmallmatrix}0&0\\0&1\end{psmallmatrix}$ as in \eqref{eq Aj graded}. This is a difference of two idempotents in $M_2((C^*_\textnormal{max}(M)^{\Gamma})^+)$, each with propagation at most $r$, and observe that 
	$$\max(\norm{p_\chi(D)},\norm{1-p_\chi(D)})\leq N.$$ 
	\begin{definition}
	\label{def qhi even}	
	Suppose $\dim M$ is even, the \emph{$(\varepsilon,r,N)$-quantitative maximal higher index} of $D$ is the class
$$\Ind_{\Gamma,\textnormal{max}}^{\varepsilon,r,N}(D)=\left[p_{\chi}(D)\right]-\begin{bmatrix}0&0\\0&1\end{bmatrix}\in K_0^{\varepsilon,r,N}(C^*_\textnormal{max}(M)^{\Gamma}).$$
	\end{definition}
	\begin{remark}

\label{rem independence even}
The class $\Ind_{\Gamma,\textnormal{max}}^{\varepsilon,r,N}(D)$ is independent of the choice of normalizing function $\chi$ satisfying \eqref{eq fourier r5}. Indeed, suppose $\chi_0$ and $\chi_1$ are two normalizing functions satisfying \eqref{eq fourier r5}. By linearity of the Fourier transform, the function 
$$\chi_t\coloneqq(1-t)\chi_0+t\chi_1$$ 
also satifies \eqref{eq fourier r5} for each $0\leq t\leq 1$. Define $p_{\chi_t}(D)$ as in \eqref{eq pchi}. Then the path $t\mapsto p_{\chi_t}(D)=(1-t)p_{\chi_0}(D)+tp_{\chi_1}(D)$ is a homotopy of $(\varepsilon,r,N)$-quasi-idempotents connecting $p_{\chi_0}(D)$ and $p_{\chi_0}(D)$. From Definition \ref{def quantitative K} (1) it follows that $A_{\chi_0}(D)$ and $A_{\chi_1}(D)$ define the same class in $K_0^{\varepsilon,r,N}(C^*_\textnormal{max}(M)^{\Gamma}).$
\end{remark}
%
\hfill\vskip 0.05in
\subsubsection*{Odd-dimensional case}
\hfill\vskip 0.05in
	\noindent Suppose $M$ is odd-dimensional. 
	We first need to make the following preparation. For each integer $n\geq 0$, define the polynomial
	$$f_n(x)\coloneqq\sum_{k=0}^n\frac{(2\pi ix)^k}{k!},$$
	and note that
	$$f_n(x)=1+\left(\sum_{k=1}^n\frac{(2\pi i)^k}{k!}\right)x+\left(\sum_{k=2}^n\frac{(2\pi i)^k}{k!}\sum_{j=0}^{k-2}x^j\right)(x^2-x).$$
	Letting
	\begin{align}
	\label{eq gn}
	g_n(x)&\coloneqq f_n(x)-\left(\sum_{k=1}^n\frac{(2\pi i)^k}{k!}\right)x^2\nonumber\\
	&=1+\left(\sum_{k=1}^n\frac{(2\pi i)^k}{k!}\right)(x-x^2)+\left(\sum_{k=2}^n\frac{(2\pi i)^k}{k!}\sum_{j=0}^{k-2}x^j\right)(x^2-x),
	\end{align}
	we see that, as $n\to\infty$, the difference
	$$e^{2\pi ix}-g_n(x)=\left(\sum_{k=n+1}^\infty\frac{(2\pi i)^k}{k!}\right)(x-x^2)+\left(\sum_{k=n+1}^\infty\frac{(2\pi i)^k}{k!}\sum_{j=0}^{k-2}x^j\right)(x^2-x)$$
	converges uniformly to $0$ for $x\in[-2,2]$. 
	Let $m=m(\varepsilon,N)$ be the least natural number such that
	\begin{align}
	\label{eq gm}
	 |g_m(x)g_m(-x)-1|<\varepsilon,\nonumber\\
	 |e^{2\pi i x}-g_m(x)|<1,
	\end{align}
	for all $x\in[-2,2]$. Now let $\chi$ be a normalizing function satisfying
	\begin{equation}
	\label{eq fourier 1}
	\supp\widehat{\chi}\subseteq\left[-\frac{r}{\deg g_m},\frac{r}{\deg g_m}\right],
	\end{equation}
	as well as $\norm{\chi}_\infty\leq 2$. Then the operator 
	$$S_\chi=\frac{\chi(D)+1}{2}$$ 
	has propagation at most $\frac{r}{\deg g_m}$ and spectrum contained in $[-\frac{1}{2},\frac{3}{2}]$.
	 Letting $A_{\chi}(D)=e^{\pi i(\chi+1)}(D)$ as in \eqref{eq Aj ungraded}, and using \eqref{eq gm}, we see that
	 \begin{equation}
	 	\norm{g_m(S_\chi)g_m(-S_\chi)-1}<\varepsilon,\quad
	 	\propagation(g_m(S_\chi))\leq r,\quad
		 \norm{g_m(\pm S_\chi)}\leq N,
	 \end{equation}
and that
\begin{equation}
\label{eq <1}
\norm{A_{\chi}(D)-g_m(S_\chi)}<1.
\end{equation}
Meanwhile, since
%
%
	$$S_\chi^2-S_\chi=\frac{\chi(D)^2-1}{4}\in C^*_\textnormal{max}(M)^{\Gamma},$$ 
	it follows from \eqref{eq gn} and Definition \ref{def quasi}
%
	that $g_m(S_\chi(D))$ is an $(\varepsilon,r,N)$-quasiinvertible in $(C^*_\textnormal{max}(M)^{\Gamma})^+$, and that
	$$\Ind_{\Gamma,\textnormal{max}}(D)=[A_\chi(D)]=[g_m(S_\chi)]\in K_1(C^*_\textnormal{max}(M)^{\Gamma}),$$
	where the second equality follows from \eqref{eq <1}. 
\begin{definition}
	\label{def qhi odd}
	For $\dim M$ odd, the \emph{$(\varepsilon,r,N)$-quantitative maximal higher index} of $D$ is the class
	$$\Ind_{\Gamma,\textnormal{max}}^{\varepsilon,r,N}(D)=[g_m(S_\chi)]\in K_1^{\varepsilon,r,N}(C^*_\textnormal{max}(M)^{\Gamma}),$$
	where $g_m = g_{m(\varepsilon,N)}$ is the polynomial defined above.
\end{definition}

\begin{remark}
\label{rem independence odd}
	The class $\Ind_{\Gamma,\textnormal{max}}^{\varepsilon,r,N}(D)$ is independent of the choice of normalizing function $\chi$ above. To see this, suppose $\chi_0$ and $\chi_1$ are two normalizing functions satisfying \eqref{eq fourier 1}, and let $S_{\chi_j}=\frac{\chi_j(D)+1}{2}$, $j=0,1$. The homotopy from Remark \ref{rem independence even} then induces a natural homotopy of $(\varepsilon,r,N)$-quasiinvertibles between $g_m(S_{\chi_0})$ and $g_m(S_{\chi_1})$. It follows from Definition \ref{def quantitative K} (2) that $g_m(S_{\chi_0})$ and $g_m(S_{\chi_1})$ define the same class in $K_1^{\varepsilon,r,N}(C^*_\textnormal{max}(M)^{\Gamma})$.
\end{remark}
	
\begin{remark}
If the parameters $\varepsilon$ and $N$ are clear from context, we may simply refer to $\Ind_{\Gamma,\textnormal{max}}^{\varepsilon,r,N}(D)$ as the \emph{quantitative maximal higher index of $D$ at scale $r$}.
\end{remark}

The usual maximal higher index of $D$ factors through its quantitative refinement:
$$
\Ind_{\Gamma,\textnormal{max}}(D)=\kappa\circ\Ind_{\Gamma,\textnormal{max}}(D),
$$
where $\kappa$ was defined in \eqref{eq kappa}.



\hfill\vskip 0.3in
\section{Proof of Theorem \ref{thm main}}
\label{sec main theorem}
\begin{proof}
Suppose the scalar curvature $\kappa$ is uniformly bounded below by some $c>0$. The Lichnerowicz formula \cite{Lichnerowicz} implies that
	$$D^2=\nabla^*\nabla+\frac{\kappa}{4}\geq\frac{c}{4},$$
	where $\nabla$ is the connection on the spinor bundle lifted from the Levi-Civita connection on $M$. It follows that $(-\frac{\sqrt{c}}{2},\frac{\sqrt{c}}{2})$ is a gap in the spectrum of $D$.
	
	Let $\chi$ be a normalizing function whose distributional Fourier transform $\widehat{\chi}$ is supported on some finite interval $[-s,s]$ for $s>0$. 
For each $t>0$, let $\chi_t$ be the normalizing function defined by
\begin{equation}
\label{eq chit}
	\chi_t(u)=\chi(tu),
\end{equation}
	$u\in\mathbb{R}$. 
Let $A_{\chi}(D)$ be the index representative defined using $\chi$. 

Consider first the case when $M$ is even-dimensional, where $A_{\chi}(D)$ is an idempotent. Denote by 
	\begin{equation*}
	A_{\chi}(u)\coloneqq
	\begin{pmatrix}
			(1-\chi(u)^2)^2 & \,\,\chi(t)(1-\chi(u)^2)\\[1ex]
			\chi(u)(2-\chi(u)^2)(1-\chi(u)^2) & \,\,\chi(u)^2(2-\chi(u)^2)-1
	\end{pmatrix},\quad u\in\mathbb{R}
	\end{equation*}
	the associated matrix of functions in $M_2(C_b(\mathbb{R}))$.
	Let $u_0>0$ and a function $\alpha$ be such that
	\begin{equation}
	\label{eq AChi small}
	\norm{A_{\chi}(u)}<\frac{\varepsilon}{2N+1}
	\end{equation}
	whenever
	$$|1-\chi(u)^2|<\alpha(\varepsilon)$$
	for all $u$ such that $|u|>u_0$, where the norm of $A_{\chi}(u)$ is taken in $M_2(\mathbb{C})$.
	Note that for $N\geq 7$, \eqref{eq AChi small} also implies that $\norm{A_{\chi}^2(u)-A_{\chi}(u)}<\varepsilon$ if $|u|>u_0$.
	By \eqref{eq chit}, we have
	\begin{equation}
	\label{eq epsilonbeta}
	\Big|1-\chi_{\frac{2u_0}{\sqrt{c}}}(u)^2\Big|=\Big|1-\chi\left(\tfrac{2u_0 u}{\sqrt{c}}\right)^2\Big|<\alpha(\varepsilon)
	\end{equation}
	whenever $u\in\mathbb{R}\backslash(-\frac{\sqrt{c}}{2},\frac{\sqrt{c}}{2})$, while 
	$$\supp\Big(\widehat{\chi}_{\tfrac{2u_0}{\sqrt{c}}}(D)\Big)\subseteq\left[-\tfrac{2u_0}{\sqrt{c}}s,\tfrac{2u_0}{\sqrt{c}}s\right].$$
	It follows that $A_{\chi_{\frac{2u_0}{\sqrt{c}}}}(D)$ is an $(\varepsilon,\frac{10u_0}{\sqrt{c}}s,N)$-quasiidempotent in $M_2((C^*_{\textnormal{max}}(M)^\Gamma)^+)$ 
	with norm strictly less than $\frac{\varepsilon}{2N+1}$.
By Lemma \ref{lem htpy},
	\begin{align*}
	\Ind_{\Gamma,\textnormal{max}}^{\varepsilon,\frac{10u_0}{\sqrt{c}}s,N}(D)=
	0\in K_0^{\varepsilon,\frac{10u_0}{\sqrt{c}}s,N}(C^*_\textnormal{max}(M)^{\Gamma}).
	\end{align*}

Letting $\omega_0=10u_0 s$, we obtain $\Ind_{\Gamma,\textnormal{max}}^{\varepsilon,\frac{\omega_0}{\sqrt{c}},N}(D)=0$. For any $r\geq\frac{\omega_0}{\sqrt{c}}$, $\Ind_{\Gamma,\textnormal{max}}^{\varepsilon,r,N}(D)$ can also be represented by $A_{\chi_{\frac{2u_0}{\sqrt{c}}}}(D)$, by Remark \ref{rem independence even}. The natural homomorphism
$$K_0^{\varepsilon,\frac{\omega_0}{\sqrt{c}},N}(C^*_\textnormal{max}(M)^{\Gamma})\to K_0^{\varepsilon,r,N}(C^*_\textnormal{max}(M)^{\Gamma})$$
induced by the inclusion
$$\textnormal{Idem}_\infty^{\varepsilon,\frac{\omega_0}{\sqrt{c}},N}((C^*_\textnormal{max}(M)^{\Gamma})^+)\hookrightarrow \textnormal{Idem}_\infty^{\varepsilon,r,N}((C^*_\textnormal{max}(M)^{\Gamma})^+)$$
takes $\Ind_{\Gamma,\textnormal{max}}^{\varepsilon,\frac{\omega_0}{\sqrt{c}},N}(D)$ to $\Ind_{\Gamma,\textnormal{max}}^{\varepsilon,r,N}(D)$, whence $\Ind_{\Gamma,\textnormal{max}}^{\varepsilon,r,N}(D)=0.$

If $M$ is odd-dimensional, let $m=m(\varepsilon,N)$ and the polynomial $g_m$ be as \eqref{eq gm}. Let $\chi$ be a normalizing function satisfying \eqref{eq fourier 1}, and let
$$s=\frac{r}{\deg g_m}.$$
	Let $u_0>0$ be such that
	$$\norm{1-g_m(P_{\chi}(u))}<\frac{\varepsilon}{N}$$
	whenever $|1-\chi(u)^2|<\alpha(\varepsilon)$ holds for all $u$ such that $|u|>u_0$ or, equivalently, whenever
	\begin{equation}
	\label{eq epsilonbeta}
	\Big|1-\chi_{\frac{2u_0}{\sqrt{c}}}(u)^2\Big|=\Big|1-\chi\left(\tfrac{2u_0 u}{\sqrt{c}}\right)^2\Big|<\alpha(\varepsilon)
	\end{equation}
	for all $u\in\mathbb{R}\backslash(-\frac{\sqrt{c}}{2},\frac{\sqrt{c}}{2})$. Meanwhile, we have
	$$\supp\Big(\widehat{\chi}_{\tfrac{2u_0}{\sqrt{c}}}(D)\Big)\subseteq\left[-\tfrac{2u_0}{\sqrt{c}}s,\tfrac{2u_0}{\sqrt{c}}s\right].$$
	It follows that $g_m(P_{\chi_{\frac{2u_0}{\sqrt{c}}}}(D))$ 
	is an $(\varepsilon,\frac{2mu_0}{\sqrt{c}}s,N)$-quasiinvertible in $M_2((C^*_{\textnormal{max}}(M)^\Gamma)^+)$ satisfying	
\begin{equation}
\Big\|1-g_m(P_{\chi_{\frac{2u_0}{\sqrt{c}}}}(D))\Big\|<\frac{\varepsilon}{N}.
\end{equation}
By Lemma \ref{lem htpy},
	\begin{align*}
	\Ind_{\Gamma,\textnormal{max}}^{\varepsilon,\frac{2mu_0}{\sqrt{c}}s,N}(D)=
	0\in K_1^{\varepsilon,\frac{2mu_0}{\sqrt{c}}s,N}(C^*_\textnormal{max}(M)^{\Gamma}).
	\end{align*}
Letting $\omega_0=2mu_0 s$, we obtain $\Ind_{\Gamma,\textnormal{max}}^{\varepsilon,\frac{\omega_0}{\sqrt{c}},N}(D)=0$. For any $r\geq\frac{\omega_0}{\sqrt{c}}$, the element $\Ind_{\Gamma,\textnormal{max}}^{\varepsilon,r,N}(D)$ can also be represented by $g_m(P_{\chi_{\frac{2u_0}{\sqrt{c}}}}(D))$, by Remark \ref{rem independence even}. The natural homomorphism
$$K_1^{\varepsilon,\frac{\omega_0}{\sqrt{c}},N}(C^*_\textnormal{max}(M)^{\Gamma})\to K_1^{\varepsilon,r,N}(C^*_\textnormal{max}(M)^{\Gamma})$$
induced by the inclusion
$$GL_\infty^{\varepsilon,\frac{\omega_0}{\sqrt{c}},N}((C^*_\textnormal{max}(M)^{\Gamma})^+)\hookrightarrow GL_\infty^{\varepsilon,r,N}((C^*_\textnormal{max}(M)^{\Gamma})^+)$$
takes $\Ind_{\Gamma,\textnormal{max}}^{\varepsilon,\frac{\omega_0}{\sqrt{c}},N}(D)$ to $\Ind_{\Gamma,\textnormal{max}}^{\varepsilon,r,N}(D)$. It follows that $\Ind_{\Gamma,\textnormal{max}}^{\varepsilon,r,N}(D)=0.$
\end{proof}
\hfill\vskip 0.3in
\section{Gromov's band width conjecture}
\label{sec band width}
In this section we turn to Theorem \ref{thm main 2}. 
\subsection{Preliminaries}
%
%
To begin, let us fix the convention for the Fourier transform: 
\[  \widehat f(\xi) = \int_{\mathbb R} f(t) e^{-it\xi}dt. \]

We will need the following estimate for Dirac-type operators in the setting of the equivariant maximal Roe algebra. Let $W$ be a Riemannian manifold with bounded geometry and $\Gamma$ a finitely generated discrete group acting properly and freely on $W$. Suppose that the following assumption is satisfied:
\begin{assumption}
	\label{ass:condition}
	There exists a fundamental domain $F$ for the $\Gamma$-action on $W$ such that
	\begin{equation*}
	l(g)\rightarrow\infty\implies d(F, gF)\rightarrow\infty,
	\end{equation*}
	where $l\colon\Gamma\rightarrow\mathbb{N}$ is a fixed length function and $d$ is the Riemannian distance on $W$.
\end{assumption}
It follows from \cite[Proposition 2.14]{GXY} that the maximal $\Gamma$-equivariant Roe algebra of $W$ is well-defined for any admissible $W$-$\Gamma$-module.

\begin{lemma}\label{lm:fp}
	Let $D$ be a $\Gamma$-invariant Dirac operator acting on a bundle $E\to W$. Suppose that
	\[  D^2 \geq c^2, \]
	for some $c\geq 0$ outside a subset $Z\subset W$. For any $\varepsilon>0$, write
	\[  U_\delta  = \{ x\in W \colon d(x, Z)>\varepsilon\}. \]
\begin{enumerate}[(i)]
\item Let $f \in \mathcal S(\mathbb R)$ be an even function such that $\widehat f$ is supported in $(-r, r)$. Then for any $\delta >0$ and $\varphi\in C_0(U_{r+\delta})$, we have 
	\[ \|f(D)\varphi \|_{\mathcal{M}} \leq \|\varphi\|_{\mathcal{M}} \sup\{ |f(y)|\colon |y| \geq c\}, \]
	where the norms on both sides are taken in the multiplier algebra $\mathcal{M}(C^*_{\textnormal{max}}(W)^\Gamma)$. The same estimate holds for $\|\varphi f(D)\|_{\mathcal{M}}$.  Moreover,
	$$\|f(D)\mathbbm{1}_{U_{r+\delta}} \|_{\mathcal{M}} \leq \sup\{ |f(y)|\colon |y| \geq c\},$$
	and the same holds for $\|\mathbbm{1}_{U_{r+\delta}}f(D)\|_{\mathcal{M}}$.
	
\item Let $f \in \mathcal S(\mathbb R)$ such that $\widehat f$ is supported in $(-r, r)$. Then for any $\delta >0$ and $\varphi\in C_0(U_{2r+\delta})$, we have 
	\[ \|f(D)\varphi \|_{\mathcal{M}} \leq 2\|\varphi\|_{\mathcal{M}} \sup\{ |f(y)|\colon |y| \geq c\}, \]
	and the same holds for $\|\varphi f(D)\|_{\mathcal{M}}$. Moreover,
	$$\|f(D)\mathbbm{1}_{U_{2r+\delta}} \|_{\mathcal{M}} \leq 2\cdot\sup\{ |f(y)|\colon |y| \geq c\},$$
	and the same holds for $\|\mathbbm{1}_{U_{2r+\delta}}f(D)\|_{\mathcal{M}}$.
\end{enumerate}
\end{lemma}
\begin{proof}The argument is analogous to that used in \cite[Lemmas 2.5, 2.6]{RoePSC}.
	
	We begin by proving (i). By \cite{GXY}, we can view $D$ and $D^2$ as an unbounded symmetric operators on the Hilbert module $C^*_{\textnormal{max}}(W)^\Gamma$ over itself, with initial domain a space of smooth $\Gamma$-invariant kernels. By an analogue of Friedrich's extension theorem \cite[Appendix A]{Zhizhang}, one can extend $D^2|_{U_\delta}$ to an essentially self-adjoint and regular operator $E$ on the same Hilbert module such that
	$$E\geq c^2.$$
	A standard finite propagation argument implies that 
	\begin{equation}
	\label{eq wave}
 	\cos(tD)\varphi =  \cos(t\sqrt{E})\varphi\\
	\end{equation}
	for all $0\leq t\leq r$ and $\varphi\in C_0(U_{r+\delta})$, as an equality in $\mathcal{M}(C^*_{\textnormal{max}}(W)^\Gamma)$. Since $\widehat f$ is supported in $(-r, r)$, we have 
	\[ f(D) =  \frac{1}{2\pi} \int_{\mathbb R} \widehat f(t) \cos(tD)\,dt = \frac{1}{2\pi} \int_{-r}^r \widehat f(t) \cos(tD)\,dt,\]
	which implies that
	\[ f(D)\varphi  = f(\sqrt{E}) \varphi\]
	for all $\varphi\in C_0(U_{r+\delta})$. It now follows from standard facts about the functional calculus that the norm of $f(\sqrt{E})$ is bounded above by
\begin{align*}	
  \sup\{ |f(y)| \,\colon |y|\geq c\}.
\end{align*}
Hence we have
\begin{equation*}
	 \|f(D)\varphi \|_{\mathcal{M}} \leq \|\varphi\|_{\mathcal{M}} \sup\{ |f(y)|\colon |y| \geq c\}.
\end{equation*}
	
	This continues to hold if we replace $\varphi$ by a bounded Borel function $\psi$ on $M$ with essential support in $U_{r+\delta}$, and that we have $\norm{\psi}\leq\norm{\psi}_\infty$. In particular, taking $\psi=\mathbbm{1}_{U_{r+\delta}}$ gives
\begin{align*}
	 \|f(D)\mathbbm{1}_{U_{r+\delta}} \|_{\mathcal{M}} &\leq \sup\{ |f(y)|\colon |y| \geq c\}
\end{align*}

Let us now prove (ii). Suppose first that $f$ is odd. By the $C^*$-identity, we have
$$\|f(D)\varphi \|_{\mathcal{M}}^2\leq\|\bar\varphi\|_{\mathcal{M}}\cdot\||f|^2(D)\varphi \|_{\mathcal{M}}.$$
Now applying part (i) to the function $|f|^2$ shows that for any $\varphi\in C_0(U_{2r+\delta})$, we have
\begin{equation*}
	 \||f|^2(D)\varphi \|_{\mathcal{M}} \leq \|\varphi\|_{\mathcal{M}} \sup\{ |f(y)|^2\colon |y| \geq c\},
\end{equation*}
which implies that
\begin{equation*}
	 \|f(D)\varphi \|_{\mathcal{M}} \leq \|\varphi\|_{\mathcal{M}} \sup\{ |f(y)|\colon |y| \geq c\}.
\end{equation*}
Similar to part (i), we also have
\begin{align*}
	 \|f(D)\mathbbm{1}_{U_{2r+\delta}} \|_{\mathcal{M}} &\leq \sup\{ |f(y)|\colon |y| \geq c\}.
\end{align*}
Combining the estimates for even and odd functions yields the stated results for general functions $f$. 
\end{proof}
%
%

In later arguments, we will use approximations of the sign function
\begin{equation}
\label{eq special chi}
\textnormal{sgn}(t) = \begin{cases}
1 & \textup{ if } t\geq 0,\\
-1 & \textup{ if } t<0
\end{cases}
\end{equation}
with compactly supported Fourier transform. It will be convenient to work with the following two specific functions:
\begin{align}
\label{eq h1}
h_{\text{1}}(t) = \frac{315}{151}\int_\mathbb{R}\operatorname{sinc}(\pi(t-s))^8\cdot\textnormal{sgn}(s)\,ds,
\end{align}
\begin{align}
\label{eq h2}
h_{\text{2}}(t) = \frac{4}{3}\int_\mathbb{R}\operatorname{sinc}(\pi(t-s))^3\cdot\textnormal{sgn}(s)\,ds,
\end{align}
where
$$\operatorname{sinc}(t)\coloneqq\frac{\sin(t)}{t}.$$
Note that $h_1$ and $h_2$ are normalizing functions satisfying
\begin{itemize}
\item $\supp\widehat{h}_1\subseteq[-8\pi,8\pi];$
\item $\supp\widehat{h}_2\subseteq[-3\pi,3\pi].$
\end{itemize}

\vspace{0.1in}
\subsection{Quantitative vanishing of the index}
We now proceed to the proof of Theorem \ref{thm main 2}. Let $M$ be a closed spin manifold of odd dimension $n-1$ with Dirac operator $D_M$ and fundamental group $\Gamma$. Let $\widetilde{M}$ be its universal cover and $D_{\widetilde{M}}$ be the lift of $D_M$ to $\widetilde M$. Unless indicated otherwise, all norms will be taken in $C^*_{\textnormal{max}}(M)^\Gamma$, $(C^*_{\textnormal{max}}(M)^\Gamma)^+$, or an associated matrix algebra.

As in the hypothesis of Theorem \ref{thm main 2}, assume that the maximal higher index of $D_{\widetilde{M}}$ does not vanish:
$$\Ind_{\Gamma,\textnormal{max}} D_{\widetilde{M}}\neq 0\in K_{1}(C^*_{\textnormal{max}}(\Gamma)).$$
 Suppose $V=M\times[-1,1]$ is equipped with a Riemannian metric $g_V$. Then $V$ is an instance of a \emph{Riemannian band}, in the sense of \cite[Section 2]{Gromov18}. Lift $g_V$ to a $\Gamma$-invariant metric $g_{\widetilde{V}}$ on the universal cover $\widetilde{V}=\widetilde{M}\times[-1,1]$. These metrics induce distance functions $d_V$ and $d_{\widetilde{V}}$ on $V$ and $\widetilde{V}$ respectively. Define the \emph{width} of $V$ to be
 $$L\coloneqq d_V(\partial_-,\partial_+)=d_{\widetilde{V}}(\widetilde{\partial}_-,\widetilde{\partial}_+),$$
 where $\partial_{\pm}$ and $\widetilde{\partial}_{\pm}$ denote the respective boundary components $M\times\{\pm 1\}$ and $\widetilde{M}\times\{\pm 1\}$ of $V$ and $\widetilde{V}$.
 
Suppose that the scalar curvature function $\kappa_V$ of $g_V$ satisfies
\[ \kappa_V \geq \sigma \]
uniformly over $V$ for some constant $\sigma>0$, so that the scalar curvature $\kappa_{\widetilde{V}}$ of $g_{\widetilde{V}}$ satisfies the same inequality.


\begin{proof}[Proof of Theorem \ref{thm main 2}]
Pick a $\Gamma$-invariant Riemannian metric $g$ on $\widetilde M\times\mathbb{R}$ that restricts to $g_{\widetilde{V}}$ on $\widetilde{V}$. We may assume that $g$ satisfies Assumption \ref{ass:condition}, for example by letting it be a product metric outside of a cocompact set containing $\widetilde{V}$. Let $\mathcal{S}_{\widetilde{M}\times\mathbb{R}}\to\widetilde{M}\times\mathbb{R}$ denote the spinor bundle. As mentioned in Remark \ref{rem maximal Roe} that, by \cite[Proposition 2.14]{GXY}, the equivariant maximal Roe algebra $C_{\max}^\ast(\widetilde M\times \mathbb R)^\Gamma$ on the admissible $(\widetilde M\times\mathbb{R})$-$\Gamma$-module $L^2(\mathcal{S}_{\widetilde M\times\mathbb{R}})$ is well-defined. 

Let $d$ be the Riemannian distance function associated to $g$.
Define
$$d_{\widetilde{\partial}_-}\colon\widetilde{M}\times\mathbb{R}\to[0,L],\qquad x\mapsto 
\begin{cases}
d(x,\widetilde{\partial}_-) & \textnormal{if }x\in M\times[-1,\infty),\\
-d(x,\widetilde{\partial}_-) & \textnormal{if }x\in M\times(-\infty,-1) 	
 \end{cases}
$$
 be the signed distance function to the boundary component $\widetilde{\partial}_-$ of $\widetilde{V}$. After possibly smoothing $d_{\widetilde{\partial}_-}$, we may assume that $\frac{L}{2}$ is a regular value. Now let
\begin{align*}
\widetilde{M}_+&\coloneqq d_{\widetilde{\partial}_-}^{-1}[\tfrac{L}{2},\infty),\\
\widetilde{M}_-&\coloneqq d_{\widetilde{\partial}_-}^{-1}(-\infty,\tfrac{L}{2}],
 \end{align*}
$$\widetilde{M}_0\coloneqq \widetilde{M}_+\cap\widetilde{M}_-=d_{\widetilde{\partial}_-}^{-1}(\tfrac{L}{2}),$$ 
where we note that the $\Gamma$-cocompact hypersurface $\widetilde{M}_0$ is $\Gamma$-equivariantly spin and is bordant to $\widetilde{M}$. We have a decomposition 
\begin{equation*}
	\widetilde{M}\times \mathbb R = \widetilde{M}_+ \cup_{\widetilde{M}_0} \widetilde{M}_-.
\end{equation*}
Since the characteristic functions $\mathbbm{1}_{\widetilde{M}_\pm}$ have zero propagation, they are bounded multipliers of $C_{\max}^\ast(\widetilde M\times\mathbb R)^\Gamma$. Thus we obtain a $\mathbb{Z}_2$-grading
	\begin{align}
	\label{eq grading}
	\mathcal{M}(C_{\max}^\ast(\widetilde M\times \mathbb R)^\Gamma)\cdot\mathbbm{1}_{\widetilde{M}_+}\oplus \mathcal{M}(C_{\max}^\ast(\widetilde M\times \mathbb R)^\Gamma)\cdot\mathbbm{1}_{\widetilde{M}_-},
	\end{align}
on the multiplier algebra $\mathcal{M}(C_{\max}^\ast(\widetilde M\times \mathbb R)^\Gamma)$, along with induced gradings on matrix algebras over it.

Define the subsets
\begin{align}
\label{eq subsets}
  A &= \widetilde V\backslash B_{\frac{L}{3}}(\widetilde{\partial}_-\cup\widetilde{\partial}_+),\nonumber\\
  Z &= (\widetilde M\times\mathbb{R})\backslash A.
\end{align}

For any subset $E\subseteq\widetilde{M}\times\mathbb{R}$, let $C^*_{\max}(E,\widetilde{M}\times\mathbb{R})^\Gamma$ denote the norm closure, in $C^*_{\max}(\widetilde{M}\times\mathbb{R})^\Gamma$, of the $*$-subalgebra of all locally compact, finite propagation operators on $L^2(\mathcal{S}_{\widetilde{M}\times\mathbb{R}})$ whose support is contained in $B_R(E)\times B_R(E)$ for some $R$ (where $R$ may depend on the operator). There is a short exact sequence
	\begin{multline}
	0\to C^*_{\max}(\widetilde M_0,\widetilde{M}\times\mathbb{R})^\Gamma\to C^*_{\max}(\widetilde M_+,\widetilde{M}\times\mathbb{R})^\Gamma\\\oplus C^*_{\max}(\widetilde M_-,\widetilde{M}\times\mathbb{R})^\Gamma\to C^*_{\max}(\widetilde{M}\times\mathbb{R})^\Gamma\to 0,
	\end{multline}
	which induces a coarse Mayer-Vietoris sequence with connecting map
	\[  \partial_{MV}\colon K_n(C_{\max}^\ast(\widetilde M\times \mathbb R)^\Gamma) \to K_{n+1}(C^*_{\max}(\widetilde M_0,\widetilde{M}\times\mathbb{R})^\Gamma)\cong K_{n+1}(C^*_{\max}(\Gamma)), \]
	where we have used the fact that
	$$C^*_{\max}(\widetilde M_0,\widetilde{M}\times\mathbb{R})^\Gamma\cong C^*_{\max}(\widetilde M_0)^\Gamma,$$
	which is Morita equivalent to $C^*_{\textnormal{max}}(\Gamma)$. It follows from Roe's partitioned manifold index theorem in the maximal setting that
	\begin{equation}
	\label{eq Roe}
	\partial_{MV}(\Ind_{\Gamma,\textnormal{max}}(D_{\widetilde{M}\times\mathbb R})) = \Ind_{\Gamma,\textnormal{max}}(D_{\widetilde{M}_0}),
	\end{equation}
	where $D_{\widetilde{M}_0}$ is the induced Dirac operator on $\widetilde{M}_0$. 
	
	We now provide estimates of the constant $C$ independently in the cases when $n$ is even and odd.\\*[5mm]
	\textbf{Case of even $n$.} The left-hand side of \eqref{eq Roe} can be represented explicitly as follows. 
	Given a normalizing function $\chi$, let $p_{\chi}=p_{\chi}(D_{\widetilde M\times\mathbb{R}})$ be as in \eqref{eq pchi}, and let
	\begin{equation}
	\label{eq restriction}
 	p_{\chi+}\coloneqq p_\chi\cdot
 	\begin{pmatrix}\mathbbm{1}_{\widetilde{M}_+}&0\\0&\mathbbm{1}_{\widetilde{M}_+}\end{pmatrix},
 	\end{equation}
 	be its restriction to the non-negative half of $\widetilde M\times\mathbb{R}$. One verifies that $p_{\chi+}$ is an idempotent in $M_2(\mathcal{M}(C^*_{\max}(\widetilde{M}\times\mathbb{R})^\Gamma))$ modulo $M_2(C^*_{\max}(\widetilde M_0,\widetilde{M}\times\mathbb{R})^\Gamma)$, hence
	$$e^{2\pi ip_{\chi+}}\in M_2((C^*_{\max}(\widetilde M_0,\widetilde{M}\times\mathbb{R})^\Gamma)^+).$$
	We then have
	\begin{equation}
	\label{eq representative}
	\partial_{MV}(\Ind_{\Gamma,\textnormal{max}}(D_{\widetilde M\times \mathbb R}))= [e^{2\pi ip_{\chi+}}].
	\end{equation}
	
	We will show that if the band width $L$ is large enough compared to the propagation of $p_\chi$, then $e^{2\pi ip_{\chi+}}$ is homotopic to $1$ through invertibles in $M_2(C^*_{\max}(\widetilde M_0,\widetilde{M}\times\mathbb{R})^\Gamma)$, and hence $\Ind_{\Gamma,\textnormal{max}}(D_{\widetilde{M}_0})$ vanishes in $K_{1}(C_{\max}^\ast( \Gamma))$ by \eqref{eq Roe}.
Cobordism invariance of the index then implies that $\Ind_{\Gamma,\textnormal{max}}(D_{\widetilde M})=0$.

	To give a numerical estimate of how large $L$ needs to be, we need to make a choice of normalizing function. We will work with
	\begin{equation}
	\label{eq chiL even}
	\chi=\chi_L(s)\coloneqq h_1\left(\frac{L}{120\pi}s\right),
	\end{equation}
	where $h_1$ was defined in \eqref{eq h1}. Let $p_{\chi_L}$ and $A_{\chi_L}$ be defined as in \eqref{eq pchi} and \eqref{eq Aj graded} with respect to $\chi_L$. Since $\widehat{h}$ has support in $[-8\pi,8\pi]$, we have:
	\begin{itemize}
	\item $\supp\widehat{\chi}_L\subseteq[-\frac{L}{15},\frac{L}{15}]$, and $\chi_L(D_{\widetilde M\times\mathbb{R}})$ has propagation at most $\frac{L}{15}$;
	\item $p_{\chi_L}$ has propagation at most $\frac{L}{3}$.
	\end{itemize}
 
	Define
	$$q\coloneqq p_{\chi_L}\cdot\begin{pmatrix}\mathbbm{1}_Z&0\\0&\mathbbm{1}_Z\end{pmatrix}+\begin{pmatrix}0&0\\0&\mathbbm{1}_A\end{pmatrix},\qquad q_+\coloneqq q\cdot\begin{pmatrix}\mathbbm{1}_{\widetilde{M}_+}&0\\0&\mathbbm{1}_{\widetilde{M}_+}\end{pmatrix}.$$
Let $p_{\chi_L+}$ be defined as in \eqref{eq restriction}. We claim that if 
\begin{equation}
\label{eq epsilon}
\norm{q-p_{\chi_L}}<\varepsilon
\end{equation}
for some $\varepsilon>0$, then
	\begin{align}
	\label{eq q+ squared}
	\norm{q_+-p_{\chi_L+}}&<\varepsilon,\qquad\norm{q_+^2-q_+}\leq(4\norm{p_{\chi_L}}+1+\varepsilon)\varepsilon,
	\end{align}
where the norms are taken in $M_2(C^*_\textnormal{max}(\widetilde{M}\times\mathbb R)^{\Gamma})$. 

The first inequality follows by observing that 
$$q_+-p_{\chi_L+}=(q-p_{\chi_L})\cdot\begin{pmatrix}\mathbbm{1}_{\widetilde{M}_+}&0\\0&\mathbbm{1}_{\widetilde{M}_+}\end{pmatrix},$$ 
together with the fact that $\Big(\begin{smallmatrix}\mathbbm{1}_{\widetilde{M}_+}&0\\0&\mathbbm{1}_{\widetilde{M}_+}\end{smallmatrix}\Big)$ is a projection in $M_2(\mathcal{M}(C_{\max}^\ast(\widetilde M\times \mathbb R)^\Gamma))$.

To see the second inequality, note first that since $p_{\chi_L}$ is an idempotent, we have
	\begin{align}
	\label{eq q squared}
	\norm{q^2-q}&=\norm{q^2-p_{\chi_L}^2+p_{\chi_L}^2-p_{\chi_L}+p_{\chi_L}-q}\nonumber\\
	&=\norm{(q+p_{\chi_L})(q-p_{\chi_L})-p_{\chi_L}q+qp_{\chi_L}+p_{\chi_L}^2-p_{\chi_L}+p_{\chi_L}-q}\nonumber\\
	&=\norm{(q+p_{\chi_L})(q-p_{\chi_L})-p_{\chi_L}(q-p_{\chi_L})+(q-p_{\chi_L})p_{\chi_L}+p_{\chi_L}-q}\nonumber\\
	&\leq\norm{q+p_{\chi_L}}\norm{q-p_{\chi_L}}+\norm{p_{\chi_L}}\norm{q-p_{\chi_L}}+\norm{q-p_{\chi_L}}\norm{p_{\chi_L}}+\norm{p_{\chi_L}-q}\nonumber\\
	&\leq(4\norm{p_{\chi_L}}+1+\varepsilon)\varepsilon.
	\end{align}
	
Define the subsets
\begin{align*}
  A_{\pm} &= A\cap\widetilde{M}_\pm,\\
  Z_{\pm} &= Z\cap\widetilde{M}_\pm
\end{align*}
and restrictions
\begin{align*}
  q_{A_{\pm}} &= q\cdot\begin{pmatrix}\mathbbm{1}_{A_{\pm}}&0\\0&\mathbbm{1}_{A_{\pm}}\end{pmatrix},\\
  q_{Z_{\pm}}&=q\cdot\begin{pmatrix}\mathbbm{1}_{Z_{\pm}}&0\\0&\mathbbm{1}_{Z_{\pm}}\end{pmatrix}
\end{align*}
of $q$ to these subsets. 

One computes that, with respect the $\mathbb{Z}_2$-grading on $M_2(\mathcal{M}(C_{\max}^\ast(\widetilde M\times \mathbb R)^\Gamma))$ induced by the grading \eqref{eq grading}, we have
$$q^2-q=\begin{pmatrix}q_{Z_+}^2-q_{Z_+}+q_{A_+} q_{Z_+} & q_{A_+}q_{Z_-}\\q_{A_-}q_{Z_+} & q_{Z_-}^2-q_{Z_-}+q_{A_-}q_{Z_-}\end{pmatrix},$$
where we have made use of the fact that 
$$q_{Z_+}q_{Z_-}=q_{Z_-}q_{Z_+}=0$$ 
because the regions $Z_+$ and $Z_-$ are separated by a distance of $\frac{L}{3}\geq\textnormal{prop}(p_{\chi_L})$. Meanwhile, since
$$q_+^2-q_+=q_{Z_+}^2-q_{Z_+}+q_{A_+}q_{Z_+},$$
we have 
\begin{equation*}
\norm{q_+^2-q_+}\leq\norm{q^2-q}.
\end{equation*}
This proves the second inequality in $\eqref{eq q+ squared}$.
	
	From this it follows that
	\begin{align}
	\label{eq pchiL+ squared}
	\norm{p_{\chi_L+}^2-p_{\chi_L+}}&=\norm{p_{\chi_L+}^2-q_+^2+q_+^2-q_++q_+-p_{\chi_L+}}\nonumber\\
	&\begin{multlined}=\norm{(p_{\chi_L+}+q_+)(p_{\chi_L+}-q_+)-q_+p_{\chi_L+}+p_{\chi_L+}q_+\\+q_+^2-q_+^2+q_+^2-q_++q_+-p_{\chi_L+}}\end{multlined}\nonumber\\
	&\begin{multlined}\leq\norm{p_{\chi_L+}+q_+}\norm{p_{\chi_L+}-q_+}+\norm{q_+}\norm{p_{\chi_L+}-q_+}\\\qquad\quad+\norm{p_{\chi_L+}-q_+}\norm{q_+}+\norm{q_+^2-q_+}+\norm{q_+-p_{\chi_L+}}\end{multlined}\nonumber\\
	&\begin{multlined}\leq(2\norm{p_{\chi_L+}}+\varepsilon+\norm{p_{\chi_L+}}+\varepsilon+\norm{p_{\chi_L+}}\\+\varepsilon+4\norm{p_{\chi_L+}}+1+\varepsilon+1)\varepsilon\end{multlined}\nonumber\\
	&\leq(8\norm{p_{\chi_L+}}+4\varepsilon+2)\varepsilon\nonumber\\
	&\leq(8\norm{p_{\chi_L}}+4\varepsilon+2)\varepsilon.
	\end{align}
For each $k\geq 2$, we have
\begin{align*}
\|p_{\chi_L+}^k-p_{\chi_L+}\| &= \|p_{\chi_L+}^{k-2}(p_{\chi_L+}^2-p_{\chi_L+})+p_{\chi_L+}^{k-1}-p_{\chi_L+} \|\\
&\leq\norm{p_{\chi_L+}}^{k-2}\cdot\norm{p_{\chi_L+}^2-p_{\chi_L+}}+\|p_{\chi_L+}^{k-1}-p_{\chi_L+}\|\\
&\leq\norm{p_{\chi_L}}^{k-2}\cdot\norm{p_{\chi_L+}^2-p_{\chi_L+}}+\|p_{\chi_L+}^{k-1}-p_{\chi_L+}\|,
\end{align*}
thus by recursion,
\begin{align}
\label{eq kth}
\|p_{\chi_L+}^k-p_{\chi_L+}\| &\leq \|p_{\chi_L+}^2-p_{\chi_L+}\|\cdot\sum_{k=0}^{n-1} \norm{p_{\chi_L}}^{n-2-k}\nonumber\\
&\leq \|p_{\chi_L+}^2-p_{\chi_L+}\|\cdot C_{\chi_L}^{-2} \dfrac{1-C_{\chi_L}^{n}}{1-C_{\chi_L}}
\end{align}
for any $C_{\chi_L}\geq\norm{p_{\chi_L}}$, $C_{\chi_L}\neq 1$.
Noting that
$$e^{2\pi ip_{\chi_L+}}-1 =\sum_{k=1}^\infty \frac{(2\pi i)^k p_{\chi_L+}^k}{k!}= \sum_{k=1}^\infty \frac{(2\pi i)^k(p_{\chi_L+}^k-p_{\chi_L+})}{k!},$$
and applying the bounds \eqref{eq pchiL+ squared} and \eqref{eq kth}, gives
\begin{align}
\label{eq exponential bound}
\|e^{2\pi ip_{\chi_L+}}-1\|\leq(8 C_{\chi_L}+4\varepsilon+2)\varepsilon\cdot\frac{e^{2\pi C_{\chi_L}}-e^{2\pi}}{C_{\chi_L}^2(C_{\chi_L}-1)}.
\end{align}
This quantity vanishes as $\varepsilon\to 0$, whence $e^{2\pi ip_{\chi_L+}}$ is homotopic to $1$ through invertibles in $M_2(C^*_{\max}(\widetilde M_0,\widetilde{M}\times\mathbb{R})^\Gamma)$.
 
	We now claim that for any choice of $\varepsilon>0$, \eqref{eq epsilon} is satisfied for all $L$ sufficiently large. To see this, recall that by the Lichnerowicz formula \cite{Lichnerowicz}, 
	$$D_{\widetilde M\times\mathbb{R}}^2=\nabla^*\nabla+\frac{\kappa}{4},$$ 
	where $\nabla$ is the connection on $\mathcal{S}_{\widetilde M\times\mathbb{R}}$ induced by the Levi-Civita connection associated to the metric $g$ on $\widetilde M\times\mathbb{R}$. Let $A\subseteq\widetilde M\times\mathbb{R}$ be defined as in \eqref{eq subsets}, so that $\kappa\geq\sigma$ on $A$. Then as observed by Friedrich \cite{Friedrich}, we have that for all $s\in L^2(\mathcal{S}_{\widetilde{M}\times\mathbb{R}})$ supported in $A$,
	\begin{align*}
	\langle D_{\widetilde M\times\mathbb{R}}	^2 s,s\rangle&\geq\langle\nabla s,\nabla s\rangle+\frac{\sigma}{4}\langle s,s\rangle\\
	&\geq\frac{1}{n}\langle D_{\widetilde M\times\mathbb{R}}s,D_{\widetilde M\times\mathbb{R}}s\rangle+\frac{\sigma}{4}\langle s,s\rangle,
	\end{align*}
	where $\langle\,\cdot\,,\cdot\,\rangle$ is the inner product on $L^2(\mathcal{S}_{\widetilde{M}\times\mathbb{R}})$. This implies
	$$D^2_{\widetilde M\times\mathbb{R}}\geq\frac{n\sigma}{4(n-1)}$$ 
	on the set $A$. Now observe that
	\begin{equation}
	\label{eq middle estimate}
	\norm{q-p_{\chi_L}}=\left\|p_{\chi_L}\cdot\begin{pmatrix}\mathbbm{1}_A&0\\0&\mathbbm{1}_A\end{pmatrix}-\begin{pmatrix}0&0\\0&\mathbbm{1}_A\end{pmatrix}\right\|=\left\|A_{\chi_L}\cdot\begin{pmatrix}\mathbbm{1}_A&0\\0&\mathbbm{1}_A\end{pmatrix}\right\|.
	\end{equation}
	By the formula for $A_{\chi_L}$ in \eqref{eq pchi} and \eqref{eq Aj graded}, we see that \eqref{eq middle estimate} is bounded above by
	\begin{align}
	\label{eq max}
	&\qquad\max\Big\{\norm{(1-\chi_L(D_{\widetilde M\times\mathbb{R}})^2)^2\cdot\mathbbm{1}_A}+\norm{\chi_L(D_{\widetilde M\times\mathbb{R}})(1-\chi_L(D_{\widetilde M\times\mathbb{R}})^2)\cdot\mathbbm{1}_A},\nonumber\\
	&\norm{\chi_L(D_{\widetilde M\times\mathbb{R}})(2-\chi_L(D_{\widetilde M\times\mathbb{R}})^2)(1-\chi_L(D_{\widetilde M\times\mathbb{R}})^2)\cdot\mathbbm{1}_A}+\norm{(1-\chi_L(D_{\widetilde M\times\mathbb{R}})^2)^2\cdot\mathbbm{1}_A}\Big\}\nonumber\\
	&\quad\qquad\leq\max\Big\{\norm{(1-\chi_L(D_{\widetilde M\times\mathbb{R}})^2)^2\cdot\mathbbm{1}_A}+\norm{(1-\chi_L(D_{\widetilde M\times\mathbb{R}})^2)\cdot\mathbbm{1}_A},\nonumber\\
	&\quad\norm{(2-\chi_L(D_{\widetilde M\times\mathbb{R}})^2)(1-\chi_L(D_{\widetilde M\times\mathbb{R}})^2)\cdot\mathbbm{1}_A}+\norm{(1-\chi_L(D_{\widetilde M\times\mathbb{R}})^2)^2\cdot\mathbbm{1}_A}\Big\}.
	\end{align}
	Define, for any $f\in C_b(\mathbb{R})$ and $s\geq 0$,
	\begin{align}
	\label{eq b}
	\qquad b_{f}(s)\coloneqq\sup\Bigg\{&\max\Big\{(1-f(t)^2)^2+(1-f(t)^2),\nonumber\\
	&|(2-f(t)^2)(1-f(t)^2)|+(1-f(t)^2)^2\Big\}\,\,\colon\,|t|\geq s
	 \Bigg\}.
	\end{align}
	Since the scalar curvature is bounded below by $\sigma$ on $A$, we can apply Lemma $\ref{lm:fp}$ (i) with $W=\widetilde M\times\mathbb{R}$, $D=D_{\widetilde M\times\mathbb{R}}$, and $Z$ as in \eqref{eq subsets}, to see that \eqref{eq max} is bounded above by
	$$b_{\chi_L}\left(\frac{1}{2}\sqrt{\frac{n\sigma}{n-1}}\right)=b_{h_1}\left(\frac{L}{240\pi}\sqrt{\frac{n\sigma}{n-1}}\right),$$
	where we have used the definition of $\chi_L$ in \eqref{eq chiL even}.
%
	Since $\chi_L^2(t)-1\to 0$ as $t\to\infty$, this quantity tends to $0$ in the limit $L\to\infty$. Thus \eqref{eq middle estimate} holds for all sufficiently large $L$. 
	Together with \eqref{eq exponential bound}, this proves the existence of \emph{some} finite constant $C$ in the statement of the theorem.
	
We now estimate $C$ numerically. Note that $\norm{p_{\chi_L}}$ is bounded above by
	\begin{align*}
	\label{eq max}
	&\qquad\max\Big\{\norm{(1-\chi_L(D_{\widetilde M\times\mathbb{R}})^2)^2}+\norm{\chi_L(D_{\widetilde M\times\mathbb{R}})(1-\chi_L(D_{\widetilde M\times\mathbb{R}})^2)},\nonumber\\
	&\norm{\chi_L(D_{\widetilde M\times\mathbb{R}})(2-\chi_L(D_{\widetilde M\times\mathbb{R}})^2)(1-\chi_L(D_{\widetilde M\times\mathbb{R}})^2)}+\norm{(1-\chi_L(D_{\widetilde M\times\mathbb{R}})^2)^2}\Big\}\\
	&\quad\qquad\leq\sup\Big\{\max\big\{(1-f(t)^2)^2+|f(t)(1-f(t)^2)|,\\
	&\qquad\qquad\qquad\qquad\qquad|f(t)(2-f(t)^2)(1-f(t)^2)|+(1-f(t)^2)^2\big\}\Big\},
	\end{align*}
which we numerically verify is strictly less than $1.29$. Using this, the bounds \eqref{eq pchiL+ squared} and \eqref{eq exponential bound} imply that
$$\norm{p_{\chi_L+}^2-p_{\chi_L+}}\leq(12.32+4\varepsilon)\varepsilon,$$
$$\norm{e^{2\pi i p_{\chi_L+}}-1}\leq\frac{e^{2\pi*1.29}-e^{2\pi}}{1.29^3-1.29^2}(12.32+4\varepsilon)\varepsilon.$$
The latter quantity is strictly less than $1$ if $\varepsilon<1.4108\times 10^{-5}.$
Thus we wish to find $C$ such that
	\begin{equation}
	\label{eq bound}
	b_{h_1}\left(\frac{L}{240\pi}\sqrt{\frac{n\sigma}{n-1}}\right)\leq 1.4108\times 10^{-5}
	\end{equation}
whenever $L > C\sqrt{\frac{n-1}{n}}\frac{1}{\sqrt{\sigma}}$, or, equivalently, such that
	\begin{equation}
	\label{eq bound 2}
	b_{h_1}\left(s\right)\leq 1.4108\times 10^{-5}
	\end{equation}
for all $s>\frac{C}{240\pi}$. One verifies numerically that \eqref{eq bound 2} holds whenever $s>0.7888$, which gives a constant $C$ of at most $190\pi$ when $n$ is even.
\\*[5mm]
	\textbf{Case of odd $n$.} In this case, the left-hand side of \eqref{eq Roe} can be written explicitly as follows. Let $U$ be any representative of the higher index of $\Ind_{\Gamma,\textnormal{max}}(D_{\widetilde M\times \mathbb R})$. Let
	\begin{equation}
	\label{eq restriction}
 	U_+\coloneqq U\cdot\mathbbm{1}_{\widetilde{M}_+}.
 	\end{equation}
 	Then $U_+$ is invertible modulo $C^*_{\max}(\widetilde M_0,\widetilde{M}\times\mathbb{R})^\Gamma$. The image $\partial_{MV}(\Ind_{\Gamma,\textnormal{max}}(D_{\widetilde M\times \mathbb R}))$ can be represented by a difference of two idempotents, as given by the formula in Definition \ref{def connectingmaps}, and which vanishes if $U_+$ is invertible in $C^*_{\max}(\widetilde{M}\times\mathbb{R})^\Gamma$.

	We now show that when the width $L$ is sufficiently large, $U_+$ can be approximated by an invertible element. Equation \eqref{eq Roe} then implies that
\begin{align}
\label{eq second last step odd}
\Ind_{\Gamma,\textnormal{max}}(D_{\widetilde M_0}) 
&= 0\in  K_{0}(C_{\max}^\ast( \Gamma)),
\end{align}
and we conclude, by cobordism invariance of the index, that $\Ind_{\Gamma,\textnormal{max}}(D_{\widetilde M})=0$. 

To estimate the constant $C$, we will use the normalizing function
\begin{equation*}
	\label{eq chiL odd}
	\chi_L(s)\coloneqq h_2\left(\frac{L}{306\pi}s\right),
\end{equation*}
where $h_2$ was defined in \eqref{eq h2}. Notice that $\chi_L(D_{\widetilde M\times \mathbb R})$ has propagation at most $\frac{L}{102}$. Let $A_{\chi_L}(D_{\widetilde M\times \mathbb R})$ be the unitary representing the index of $D_{\widetilde M\times\mathbb R}$ defined using $\chi_L$, and let the subsets $A$ and $Z$ be as in \eqref{eq subsets}. Suppose as before that $\kappa\geq\sigma$ on $\widetilde{V}$.

Define
\begin{align*}
S_{\chi_L}&\coloneqq\tfrac{\chi_L(D_{\widetilde{M}\times\mathbb{R}})+1}{2},\\
u_{\chi_L}&\coloneqq g_{17}(S_{\chi_L}),\\
v_{\chi_L}&\coloneqq g_{17}(-S_{\chi_L}),
\end{align*}
where the polynomial $g_{17}$ was defined in \eqref{eq gn}. Then $u_{\chi_L}$ and $v_{\chi_L}$ have propagations at most $\frac{L}{6}$. From the identity
$$\Big|e^z-\sum_{j=0}^d\frac{z^j}{j!}\Big|\leq\max\{1,e^{\operatorname{Re}(z)}\}\cdot\frac{|z|^d}{d!}$$
for all $z\in\mathbb{C}$, we have\\[-15pt]
\begin{align*}
\delta&\coloneqq\max\left\{\norm{A_{\chi_L}(D_{\widetilde M\times \mathbb R})-u_{\chi_L}},\norm{A_{\chi_L}(D_{\widetilde M\times \mathbb R})^{-1}-v_{\chi_L}}\right\}\\[10pt]
&\leq\bigg\|\sum_{k=0}^{17}\frac{(2\pi iS_{\chi_L})^k}{k!}\bigg\| + \bigg\|\sum_{k=1}^{17}\frac{(2\pi i)^k}{k!}S_{\chi_L}^2\bigg\|\\[10pt]
&\leq\frac{(2\pi)^{17}(\norm{S_{\chi_L}}^{17}+\norm{S_{\chi_L}}^2)}{17!},\\[-15pt]
\end{align*}
where the norms are taken in $(C^*_\textnormal{max}(M)^{\Gamma})^+$. 
One finds numerically that $\delta<0.209$. 

We claim that for any $0<\varepsilon<1$, we have
	\begin{align}
	\label{eq middle estimate odd}
	\max\big\{\left\|(u_{\chi_L}-1)\cdot\mathbbm{1}_A\right\|,\left\|(v_{\chi_L}-1)\cdot\mathbbm{1}_A\right\|\big\}\leq\varepsilon
	\end{align}
for $L$ sufficiently large. Indeed, by applying Lemma $\ref{lm:fp}$ (ii) with $W=\widetilde M\times\mathbb{R}$, $D=D_{\widetilde M\times\mathbb{R}}$, and the above choice of $Z$, we see that
	\begin{equation}
	\label{eq bound odd}
	\norm{(u_{\chi_L}-1)\cdot\mathbbm{1}_A}\leq 2\cdot \sup\left\{|u_{\chi_L}(t)-1|\colon|t|\geq\frac{1}{2}\sqrt{\frac{n\sigma}{n-1}}\right\},
	\end{equation}
	where $u_{\chi_L}(t)$ is the function $t\mapsto g_{17}\left(\frac{1}{2}(\chi_L(t)+1)\right)$.
	Note that $g_{17}(x)-1$ contains a factor of $x^2-x$ 	
	and that 
	$$\left(\frac{1}{2}(\chi_L(t)+1)\right)^2-\frac{1}{2}(\chi_L(t)+1)=\frac{1}{4}(\chi_L(t)^2-1).$$ 
	This, together with the fact that $\chi_L(t)^2-1\to 0$ as $t\to\pm\infty$, implies that \eqref{eq bound odd} approaches $0$ as $L\to\infty$. The same applies to $\left\|(v_{\chi_L}-1)\cdot\mathbbm{1}_A\right\|$. Thus \eqref{eq middle estimate odd} holds when $L$ is sufficiently large.
	
	Suppose \eqref{eq middle estimate odd} holds for some $\varepsilon$ and $L$. Define
	\begin{align*}
	u&\coloneqq u_{\chi_L}\cdot\mathbbm{1}_{Z\backslash A}+\mathbbm{1}_A,\\
	v&\coloneqq v_{\chi_L}\cdot\mathbbm{1}_{Z\backslash A}+\mathbbm{1}_A.
	\end{align*}
	Then $\norm{u-u_{\chi_L}}\leq\varepsilon$ and $\norm{v-v_{\chi_L}}\leq\varepsilon$, whence
	\begin{align*}
	&\begin{multlined}\norm{uv-1}=\norm{uv-uv_{\chi_L}+uv_{\chi_L}-uA_{\chi_L}(D_{\widetilde M\times \mathbb R})^{-1}+uA_{\chi_L}(D_{\widetilde M\times \mathbb R})^{-1}\\-u_{\chi_L}A_{\chi_L}(D_{\widetilde M\times \mathbb R})^{-1}+A_{\chi_L}(D_{\widetilde M\times \mathbb R})A_{\chi_L}(D_{\widetilde M\times \mathbb R})^{-1}-1}\end{multlined}\nonumber\\
	&\qquad\qquad\leq\norm{u}\varepsilon+\norm{u}\delta+\varepsilon\nonumber\\
	&\qquad\qquad=\norm{u}(\delta+\varepsilon)+\varepsilon\nonumber\\
	&\qquad\qquad\leq(1+\delta)(\delta+\varepsilon)+\varepsilon.
	\end{align*}
	and similarly for $\norm{vu-1}$.
	We also have
\begin{align*}
\norm{A_{\chi_L}(D_{\widetilde M\times \mathbb R})-u}\leq\norm{A_{\chi_L}(D_{\widetilde M\times \mathbb R})-u_{\chi_L}}+\norm{u-u_{\chi_L}}\leq\delta+\varepsilon,
\end{align*}
and similarly for $\norm{A_{\chi_L}(D_{\widetilde M\times \mathbb R})^{-1}-v}$.
	Write 
	\begin{align*}
	u_+&=u\cdot\mathbbm{1}_{\widetilde{M}_+},\qquad v_+=v\cdot\mathbbm{1}_{\widetilde{M}_+}.
	\end{align*}
An argument analogous to that used to derive \eqref{eq q+ squared} in the even case shows that
	\begin{align}
	\label{eq v+}
	\max\{\norm{u_+v_+-1},\norm{v_+u_+-1}\}&\leq(1+\delta)(\delta+\varepsilon)+\varepsilon.
	\end{align}
	One verifies that if $\varepsilon\leq 0.338$, then both
	$(1+\delta)(\delta+\varepsilon)+\varepsilon$ and $\delta+\varepsilon$ are strictly less than $1$. This implies that	
$$\Ind_{\Gamma,\textnormal{max}}(D_{\widetilde{M}\times\mathbb{R}})=[A_\chi(D_{\widetilde{M}\times\mathbb{R}})]=[u]\in K_1(C^*_\textnormal{max}(\widetilde{M}\times\mathbb{R})^{\Gamma}),$$
	and that $u_+$ is invertible. 
	

Thus we wish to find $C$ such that 
	\begin{equation}
	\label{eq bound}
	\sup\left\{|u_{\chi_L}(t)-1|\colon|t|\geq\frac{1}{2}\sqrt{\frac{n\sigma}{n-1}}\right\}\leq 0.169
	\end{equation}
whenever $L>C\sqrt{\frac{n-1}{n}}\frac{1}{\sqrt{\sigma}}.$
By definition of $\chi_L$, \eqref{eq bound} is equivalent to
\begin{equation*}
\sup\left\{|u_{h_2}(t)-1|\colon|t|\geq\frac{L}{612\pi}\sqrt{\frac{n\sigma}{n-1}}\right\}\leq 0.169.
\end{equation*}
It can be verified numerically that $|u_{h_2}(t)-1|$
is bounded above by $0.169$ for all $t$ satisfying $|t|\geq 0.535$. This gives a constant $C$ of at most $328\pi$ when $n$ is odd.
\end{proof}

\hfill\vskip 0.3in
\bibliographystyle{plain}
\bibliography{mybib}
\end{document}